\documentclass{amsart}
\usepackage{amssymb, enumerate}
\usepackage{amsrefs}
\usepackage{graphicx}
\usepackage{hyperref}
\usepackage{color}

\theoremstyle{plain}
\newtheorem{theorem}{Theorem}
\newtheorem{lemma}[theorem]{Lemma}

\newtheorem{prop}[theorem]{Proposition}
\newtheorem{corollary}[theorem]{Corollary}

\theoremstyle{definition}
\newtheorem{definition}[theorem]{Definition}

\theoremstyle{remark}
\newtheorem{remark}[theorem]{Remark}

\newcommand{\lp}{\left(}

\newcommand{\rp}{\right)}
\newcommand{\e}{\epsilon}
\newcommand{\R}{\mathbb{R}}
\newcommand{\C}{\mathbb{C}}
\newcommand{\Q}{\mathbb{Q}}
\renewcommand{\O}{\mathbb{O}}
\newcommand{\G}{\mathbb{G}}

\renewcommand{\H}{\mathbb{H}}

\newcommand{\Hy}{\mathbb{H}}

\newcommand{\Sp}{\mathbb{S}}

\newcommand{\acts}{\curvearrowright}

\newcommand{\N}{\mathbb{N}}

\newcommand{\ti}{\textit}
\newcommand{\grad}{\nabla}

\DeclareMathOperator{\diam}{\textup{\text{diam}}}

\DeclareMathOperator{\CAT}{\textup{\text{CAT}}}

\DeclareMathOperator{\cdim}{\textup{conf.dim}}
\DeclareMathOperator{\AR}{\textup{AR}}
\DeclareMathOperator{\Isom}{\textup{Isom}}

\DeclareMathOperator{\Hom}{\textup{Hom}}
\DeclareMathOperator{\Tan}{\textup{Tan}}
\DeclareMathOperator{\Hdim}{\textup{Hdim}}

\numberwithin{equation}{section}
\numberwithin{theorem}{section}

\begin{document}

\title[Rigidity for convex-cocompact actions]{Rigidity for convex-cocompact actions on rank-one symmetric spaces}
\author{Guy C. David}
\address{Courant Institute of Mathematical Sciences, New York University, New York, NY 10012}
\email{guydavid@math.nyu.edu}

\author{Kyle Kinneberg}

\address{Department of Mathematics, Rice University, 6100 Main St., Houston TX 77005}
\email{kyle.kinneberg@rice.edu}

\subjclass[2010]{Primary: 53C24, 53C35; Secondary: 53C17, 53C23}
\date{\today}
\keywords{convex-cocompact action, rank-one symmetric space, Carnot group}

\begin{abstract}
When $\Gamma \acts X$ is a convex-cocompact action of a discrete group on a non-compact rank-one symmetric space $X$, there is a natural lower bound for the Hausdorff dimension of the limit set $\Lambda(\Gamma) \subset \partial X$, given by the Ahlfors regular conformal dimension of $\partial \Gamma$. We show that equality is achieved precisely when $\Gamma$ stabilizes an isometric copy of some non-compact rank-one symmetric space in $X$ on which it acts with compact quotient. This generalizes a theorem of Bonk--Kleiner, who proved it in the case that $X$ is real hyperbolic.

To prove our main theorem, we study tangents of Lipschitz differentiability spaces that are embedded in a Carnot group $\G$. We show that almost all tangents are isometric to a Carnot subgroup of $\G$, at least when they are rectifiably connected. This extends a theorem of Cheeger, who proved it for PI spaces that are embedded in Euclidean space.
\end{abstract}

\maketitle

\section{Introduction}\label{intro}

A classic theorem of Bowen \cite{Bo79} states that the limit set of a convex-cocompact action of a Fuchsian group on $\H_\R^3$ has Hausdorff dimension at least 1, and equality holds precisely when the limit set is a round circle in $\Sp^2 = \partial \H_\R^3$. In this case, the convex-cocompact action stabilizes an isometrically embedded copy of $\H_\R^2$ in $\H_\R^3$. Expressed another way, this result means that among all quasi-Fuchsian representations of a Fuchsian group, the Hausdorff dimension of the limit set attains its minimum precisely on the Fuchsian locus.

Bowen's theorem was later generalized to convex-cocompact actions of Fuchsian groups on any $\CAT(-1)$ metric space by Bonk and Kleiner \cite[Theorem 1.1]{BKqf}. This answered a question of Bourdon, who had proven that the analogous result holds for any uniform lattice in a non-compact rank-one symmetric space $S \neq \Hy_\R^2$ \cite[Theorem 0.3]{Bou}. To state their results precisely, we need some notation. If $X$ is a $\CAT(-1)$ space, then its visual boundary $\partial X$ admits a natural class of M\"obius equivalent metrics 
$$d_p(x,y) = e^{-(x,y)_p},$$ 
where $p \in X$ and $(x,y)_p$ denotes the Gromov product based at $p$. In particular, these metrics are bi-Lipschitz equivalent. The Hausdorff dimension of $\partial X$, and also of any subset, is therefore well-defined. If $\Gamma \acts X$ is a convex-cocompact action by a discrete group $\Gamma$, then the limit set $\Lambda(\Gamma)$ of the action is a subset of $\partial X$ that is invariant under the induced boundary action on $\partial X$.

The following statement puts together the results of Bourdon (all cases $S \neq \Hy_\R^2$) and Bonk--Kleiner (the $S=\Hy_\R^2$ case).

\begin{theorem}[\cite{Bou}, \cite{BKqf}] \label{Bouthm}
Let $\Gamma$ be a uniform lattice in a non-compact rank-one symmetric space $S$, and suppose that $\Gamma \acts X$ is a convex-cocompact action on a $\CAT(-1)$ space $X$. Then
$$\Hdim(\Lambda(\Gamma)) \geq \Hdim(\partial S),$$ 
and equality holds if and only if the action stabilizes an isometrically embedded copy of $S$ in $X$ on which $\Gamma$ acts with compact quotient.
\end{theorem}

We should note that the case of equality is also characterized by M\"obius equivalence of the limit set $\Lambda(\Gamma)$ and $\partial S$ \cite[Theorem 0.1]{Bou}.

\begin{remark}
The notation $\Lambda(\Gamma)$ is slightly misleading, as the limit set depends on the action of $\Gamma$ on $X$, not just on the group $\Gamma$. More accurately, one should think of the action as a representation of $\Gamma$ into the isometry group $\Isom(X)$, so the limit set is really the limit set of this representation. In many important cases, the space of convex-compact representations forms a moduli space with a rich geometric structure, and one can ask interesting questions about how the limit sets change as the representations vary (e.g. \cite[Corollary 1.8]{BCLS}). However, as we focus only on a single action/representation at a time, we do not make this distinction explicit, and we retain the notation $\Lambda(\Gamma)$.
\end{remark}

More recently, there has been interest in knowing whether similar types of rigidity phenomena occur when $\Gamma$ is not assumed, a priori, to be a lattice \cites{BKqm, Kin}. The first difficulty encountered is to determine a natural lower bound for $\Hdim(\Lambda(\Gamma))$. For example, free groups admit convex-cocompact actions on any non-compact rank-one symmetric space with $\Lambda(\Gamma)$ equal to a Cantor set of arbitrarily small Hausdorff dimension. More explicitly, choosing loxodromic elements $g_1, \ldots ,g_k \in \Isom(S)$ with pairwise disjoint fixed points, for large enough $n$ the subgroup generated by $g_1^n,\ldots,g_k^n$ will be a free group of rank $k$, and the Hausdorff dimension of the limit set will tend to 0 as $n$ goes to infinity.

In the case of Theorem \ref{Bouthm}, the inequality comes from a result of Pansu \cite[Theorem 5.5]{Pancd}. In standard formulation, Pansu's result says that the Hausdorff dimension of any metric space that is quasisymmetrically equivalent to $\partial S$ must be at least as large as $\Hdim(\partial S)$. Stated another way, this means that whenever $\Gamma$ is (virtually) a uniform lattice in $S$, the Ahlfors regular conformal dimension $\cdim_{\AR}(\partial \Gamma)$ of its Gromov boundary is equal to $\Hdim(\partial S)$. On the other hand, we note that $\cdim_{\AR}(\partial \Gamma) = 0$ whenever $\Gamma$ is a free group.

It makes sense, then, that the appropriate generalization of Theorem \ref{Bouthm} to Gromov hyperbolic groups would make use of the conformal dimension. In a different paper \cite{BKconf}, Bonk and Kleiner established the following result, which concerns actions of hyperbolic groups on real hyperbolic space.

\begin{theorem}[\cite{BKconf}, Theorem 1.4]\label{BKthm}
Let $\Gamma$ be a non-elementary Gromov hyperbolic group, and let $Q = \cdim_{\AR}(\partial \Gamma)$. If $\Gamma \acts \Hy_\R^{n+1}$ is a convex-cocompact action, with $n \geq 1$, then
$$\Hdim(\Lambda(\Gamma)) \geq Q$$
and equality holds if and only if $Q = k \geq 1$ is an integer and $\Gamma$ stabilizes an isometric copy of $\H_\R^{k+1}$ in $\H_\R^{n+1}$ on which it acts with compact quotient.
\end{theorem}

Our statement of this theorem is slightly different from that given in \cite{BKconf}, but they are easily seen to be equivalent. We make this clear in Remark \ref{tfaermk}.

The main goal of the present paper is to extend Theorem \ref{BKthm} to actions on any non-compact rank-one symmetric space. 

\begin{theorem} \label{mainthm}
Let $\Gamma$ be a non-elementary Gromov hyperbolic group, and let $Q = \cdim_{\AR}(\partial \Gamma)$. If $\Gamma \acts X$ is a convex-cocompact action on a non-compact rank-one symmetric space $X$, then
$$\Hdim(\Lambda(\Gamma)) \geq Q,$$
and equality holds if and only if $\Gamma$ stabilizes an isometric copy of a non-compact rank-one symmetric space $S$ in $X$ on which it acts with compact quotient.
\end{theorem}

In the equality case, $\partial S$ and $\Lambda(\Gamma)$ are M\"obius equivalent, $Q = \Hdim(\partial S)$ is necessarily an integer, and the symmetric space $S$ is determined uniquely by $Q$ and the topological dimension of $\partial \Gamma$. Indeed, the non-compact rank-one symmetric spaces are uniquely determined by the Hausdorff and topological dimensions of their visual boundaries (cf. \cite[p. 34]{MTcdim}). Moreover, $\Gamma$ is a finite extension of a uniform lattice in $\Isom(S)$.

Our proof of Theorem \ref{mainthm} follows the basic outline that Bonk--Kleiner use to prove Theorem \ref{BKthm}. First, the inequality comes straight from the general theory of boundaries of hyperbolic groups. Appealing to the prior work of Bonk--Kleiner in \cite{BKqf} also dispenses with the equality case when $Q \leq 1$. When $Q > 1$ and equality holds, Bonk and Kleiner show that the limit set $\Lambda(\Gamma) \subset \Sp^n$ supports a Poincar\'e inequality (in the sense of \cite{HK98}). As $\Lambda(\Gamma)$ is isometrically embedded in $\R^{n+1}$, a theorem of Cheeger \cite{Ch99} guarantees that $\Lambda(\Gamma)$ has a tangent that is isometric to some Euclidean space $\R^k$. This implies that $\Lambda(\Gamma)$ is bi-Lipschitz equivalent to $\Sp^k$, and they can use \cite[Theorem 1.1]{BKqm} to conclude.

In the setting of Theorem \ref{mainthm}, when $Q >1$ and equality holds, the work of Bonk--Kleiner still implies that the limit set $\Lambda(\Gamma)$ supports a Poincar\'e inequality. However, it may not be embedded in any Euclidean space, so we cannot apply Cheeger's theorem. Instead, $\Lambda(\Gamma)$ is locally embedded in some Carnot group, namely, the Carnot group that models the local geometry of $\partial X$. The main work we do is to generalize Cheeger's result to spaces embedded in Carnot groups, showing that their tangents are isometric to Carnot subgroups. This is done in Theorem \ref{tan} below, which may be of independent interest. The remainder of the proof follows similar ideas as in the endgame of \cite[Theorem 1.2]{BKqm}, though we will need an additional input that identifies, among all Carnot groups, those that locally model boundaries of rank-one symmetric spaces. This is provided by recent work of Cowling and Ottazzi \cite{CO15}.

The structure of the paper is as follows. Section \ref{prelims} contains basic definitions and notation. Section \ref{symgromov} contains background material on boundaries of symmetric spaces and Gromov hyperbolic groups, as well as an analog of a result of Connell \cite{Con} that we prove as a short application of our methods (Proposition \ref{connell}). Section \ref{carnottangent} contains the proof of Theorem \ref{tan} about tangents of spaces embedded in Carnot groups, and describes some applications to bi-Lipschitz non-embedding. Finally, Section \ref{mainproof} contains the proof of Theorem \ref{mainthm}.

\section{Preliminaries}\label{prelims}
If $(X,d)$ is a metric space, we write $B(x,r)$ for the open ball of radius $r$ centered at $x\in X$ and $\overline{B}(x,r)$ for the closed ball. A metric space $(X,d)$ is said to be \ti{metrically doubling} if there is a constant $C$ such that every ball of radius $r >0 $ in $X$ can be covered by at most $C$ balls of radius $r/2$. 

A \ti{metric measure space} $(X,d,\mu)$ is a complete, separable metric space $(X,d)$ equipped with a Radon measure $\mu$. We say that the measure $\mu$ is \ti{doubling} if there is a constant $C$ such that
$$\mu(B(x,2r)) \leq C \mu(B(x,r))$$
for every ball $B(x,r)$ in $X$. This readily implies that $(X,d)$ is metrically doubling \cite{He01}.

For $Q>0$, we will often consider $Q$-dimensional Hausdorff measure $\mathcal{H}^Q$ on a metric space $(X,d)$. We write $\Hdim(X)$ for the Hausdorff dimension of $X$. A metric space is \ti{Ahlfors $Q$-regular}, for $Q>0$, if there is a constant $C>0$ such that
$$ C^{-1}r^Q \leq \mathcal{H}^Q(\overline{B}(x,r)) \leq Cr^Q $$
for all $x\in X$ and $0<r\leq \diam(X)$. In this case, $(X,d,\mathcal{H}^Q)$ is easily seen to be a doubling metric measure space. It is also uniformly perfect \cite[pp. 13-14]{MTcdim}.

The \ti{cross-ratio} of a four-tuple of distinct points $(x_1, x_2, x_3,x_4)$ in $X$ is the value
$$ [x_1, x_2, x_3, x_4] = \frac{d(x_1,x_3) d(x_2,x_4)}{d(x_1,x_4) d(x_2,x_3)}.$$
A homeomorphism $f\colon X\rightarrow Y$ between two metric spaces is \ti{M\"obius} if it preserves the cross-ratio, i.e., if
$$ [f(x_1), f(x_2), f(x_3), f(x_4)] = [x_1, x_2, x_3, x_4]$$
for every four-tuple  $(x_1,x_2,x_3,x_4)$ of distinct points in $X$.

More generally, given a homeomorphism $\eta\colon [0,\infty)\rightarrow[0,\infty)$, we say that a homeomorphism $f\colon X\rightarrow Y$ is \ti{$\eta$-quasi-M\"obius} (or simply \ti{quasi-M\"obius}) if
$$ [f(x_1), f(x_2), f(x_3), f(x_4)] \leq \eta([x_1, x_2, x_3, x_4])$$
for every four-tuple  $(x_1,x_2,x_3,x_4)$ of distinct points in $X$. Similarly, $f$ is said to be \ti{$\eta$-quasisymmetric} (or simply \ti{quasisymmetric}) if
$$ \frac{d(f(x_1),f(x_2))}{d(f(x_1),f(x_3))} \leq \eta\left(\frac{d(x_1,x_2)}{d(x_1,x_3)}\right)$$
for every triple  $(x_1,x_2,x_3)$ of distinct points in $X$.

These definitions also make sense if $f \colon X \rightarrow Y$ is assumed only to be injective, in which case it is necessarily a homeomorphism onto its image. Such a map $f$ is called a M\"obius embedding, quasi-M\"obius embedding, etc. It is not difficult to show that every $\eta$-quasisymmetric mapping is $\tilde{\eta}$-quasi-M\"obius, where $\tilde{\eta}$ depends only on $\eta$. If $X$ and $Y$ are bounded, then, conversely, every quasi-M\"obius homeomorphism is quasisymmetric \cite{Vai84}.

In the case that $X$ and $Y$ are bounded and $\eta$ is linear, each $\eta$-quasi-M\"obius map $f\colon X\rightarrow Y$ is actually bi-Lipschitz \cite[Remark 3.2]{Kin}. In particular, every M\"obius map is bi-Lipschitz.

The \ti{Ahlfors regular conformal dimension} of a doubling and uniformly perfect metric space $X$, denoted
$$\cdim_{\AR}(X),$$
is the infimal Hausdorff dimension among all Ahlfors regular metric spaces quasisymmetrically homeomorphic to $X$. The doubling and uniformly perfect properties guarantee that $\cdim_{\AR}(X)$ is finite \cite[Corollary 14.15]{He01}, and it is obviously quasisymmetrically invariant. However, the infimum is not always achieved, and it is a difficult problem to find general conditions under which it is. For more background on this quantity, we refer the reader to \cite{BKconf}, where it plays a crucial role.

\subsection{Tangents}
We now briefly describe the notion of a ``tangent'' of a metric space, which relies on the framework of pointed Gromov-Hausdorff convergence of spaces and mappings. For the necessary background on this framework, we refer the reader to any of the following: \cite[Section 5]{Ke04}, \cite[Section 3.2]{KM11}, \cite[Section 2]{GCD15}. 

A \ti{pointed metric space} $(X,d,x)$ is simply a metric space $(X,d)$ together with a fixed point $x\in X$. If the metric $d$ is clear, we sometimes suppress it and simply write $(X,x)$.

Suppose that $(X,d)$ is a doubling metric space. Fix $x \in X$ and let $\lambda_j \rightarrow \infty$ be a sequence of positive scaling factors. The sequence of pointed metric spaces $(X,\lambda_j d, x)$ then has a subsequence that converges in the pointed Gromov-Hausdorff sense to a complete pointed metric space $(\hat{X},\hat{x})$. We say that $(\hat{X},\hat{x})$ is a \ti{tangent} of $X$ at $x$. The set of pointed isometry classes of tangents at $x$ is denoted by $\Tan(X,x)$.

If $f \colon X \rightarrow Y$ is a Lipschitz map taking values in a complete, doubling metric space $(Y,\rho)$ with $y=f(x)$, then the sequence of Lipschitz maps
$$f \colon (X,\lambda_j d, x) \rightarrow (Y, \lambda_j \rho, y)$$
has a subsequence that converges to a Lipschitz map
$$\hat{f} \colon (\hat{X},\hat{x}) \rightarrow (\hat{Y},\hat{y}),$$
where $(\hat{Y}, \hat{y}) \in \Tan(Y,y)$. We will repeatedly use the ``hat" notation for tangents and tangent maps. Note that if we first fix $(\hat{X}, \hat{x}) \in \Tan(X,x)$, then by passing to further subsequences, we may find a corresponding $(\hat{Y}, \hat{y}) \in \Tan(Y,y)$ and tangent map $\hat{f} \colon \hat{X} \rightarrow \hat{Y}$. The set of tangents and mappings obtained in this way, up to isometric equivalence, is denoted by $\Tan(X,x,f)$. 

\begin{remark}
When $Y = \R^n$ is a finite-dimensional Euclidean space, we always identify $(\hat{Y},\hat{y})$ with $(\R^n,0)$ in the canonical way, so that $\hat{f}(x) = 0$. In this way, tangent maps behave like derivatives in the classical sense.
\end{remark}

\subsection{Lipschitz differentiability spaces}
In the seminal paper \cite{Ch99}, Cheeger introduced a type of differentiable structure for real-valued Lipschitz functions on metric measure spaces. He showed that a large class of metric measure spaces, the so-called ``PI spaces'' (those which are doubling and have a Poincar\'e inequality in the sense of \cite{HK98}), support such a structure.

The following definition is due to Cheeger \cite{Ch99}.

\begin{definition}\label{LDspace}
A metric measure space $(X,d,\mu)$ is called a \textit{Lipschitz differentiability space} if it satisfies the following condition. There are countably many Borel sets (``charts'') $U_i$ covering $X$, positive integers $n_i$ (the ``dimensions of the charts''), and Lipschitz maps $\phi_i\colon X\rightarrow\mathbb{R}^{n_i}$ with respect to which any Lipschitz function $f \colon X \rightarrow \R$ is differentiable almost everywhere, in the sense that for each $i$ and for $\mu$-almost every $x\in U_i$, there exists a unique $\grad f(x)\in\mathbb{R}^{n_i}$ such that
\begin{equation} \label{LD}
\lim_{y\rightarrow x} \frac{|f(y) - f(x) - \grad f(x)\cdot(\phi_i(y)-\phi_i(x))|}{d(x,y)} = 0.
\end{equation}
Here $\grad f(x)\cdot(\phi_i(y)-\phi_i(x))$ denotes the standard scalar product in $\mathbb{R}^{n_i}$.
\end{definition}
Note that the Borel measurability of the function $x\mapsto \grad f(x)$ and the set of differentiability points of $f$ are consequences of this definition; see \cite[Remark 1.2]{BS13}. In recent years, the study of Lipschitz differentiability spaces in their own right has become an active area of research, and we refer the reader to \cites{Ch99, Ke04, CK09, KM11, Ba15, GCD15, BL15, BL16, CKS15} for more background.

If $(X,d)$ is doubling, it is natural to rephrase \eqref{LD} in terms of tangents. Letting $df(x) \in \Hom(\R^n,\R)$ be the dual functional to $\grad f(x)$, the asymptotic in (\ref{LD}) is equivalent to the following condition. If $(\hat{X},\hat{x}) \in \Tan(X,x)$ and we have associated tangent mappings $\hat{f} \colon \hat{X} \rightarrow \R$ and $\hat{\phi}\colon \hat{X}\rightarrow\R^n$ to $f$ and $\phi$, respectively, then $\hat{f}$ factors through $\R^n$ via $\hat{f} = df(x) \circ \hat{\phi}$.

A similar factorization works for any Lipschitz map $f \colon X \rightarrow \R^N$ into a Euclidean space, simply by considering the coordinate functions for $f$. Namely, if $(U,\phi)$ is an $n$-dimensional chart in $X$, then for almost every $x \in U$, there is a unique linear map $Df(x) \in \Hom(\R^n,\R^N)$ such that for every $(\hat{X}, \hat{x}) \in \Tan(X,x)$, any tangent map $\hat{f} \colon \hat{X} \rightarrow \R^N$ factors through $\R^n$ via
$$\hat{f} = Df(x) \circ \hat{\phi}.$$
Not surprisingly, the rows of the matrix for $Df(x)$ are simply the gradients $\grad f_i(x)$ of the coordinate functions $f = (f_1,\ldots,f_N)$.

Recently, there has been much interest in understanding the geometric structure of tangents of Lipschitz differentiability spaces. One result in this vein is due to Cheeger--Kleiner--Schioppa \cite{CKS15}, and it will be important for us below. First, another piece of terminology.

A Lipschitz map $f \colon X \rightarrow Y$ between two metric spaces is called a \ti{Lipschitz quotient} if it is Lipschitz and there exists $c >0$ such that
$$B(f(x),cr) \subset f(B(x,r))$$
for all $x \in X$ and $r >0$. The supremum of such constants $c$ is called the \ti{co-Lipschitz} constant of $f$. Note that a Lipschitz quotient map is clearly surjective.

We say that $f$ is a \ti{metric submersion} if it is both $1$-Lipschitz and has co-Lipschitz constant $c=1$.

\begin{prop}[\cite{CKS15}, Theorem 1.12] \label{LQ}
Let $(X,d)$ be a complete, doubling metric space, and let $\mu$ be a Radon measure on $X$ such that $(X,d,\mu)$ is a Lipschitz differentiability space. Let $(U,\phi)$ be an $n$-dimensional chart in $X$.

Then for $\mu$-almost every $x \in U$, there is a norm $\|\cdot\|_x$ on $\R^n$ with the following property. For every $(\hat{X}, \hat{x}) \in \Tan(X,x)$ with tangent map $\hat{\phi} \colon \hat{X} \rightarrow \R^n$, the mapping $\hat{\phi}$ is a metric submersion onto $(\R^n, \|\cdot\|_x)$.

In particular, the mapping $\hat{\phi}$ is a Lipschitz quotient onto $\R^n$ with its standard Euclidean metric.
\end{prop}
The statement about Lipschitz quotients in Proposition \ref{LQ} already appeared in earlier (independent) work of Schioppa (Theorem 5.56 in \cites{Sc13, Sc16}) and the first-named author \cite[Corollary 5.1]{GCD15}. This statement alone is enough for the application to our main result, Theorem \ref{mainthm}. We state the stronger version about metric submersions only for the applications in Section \ref{consec} below.
	
As Cheeger observed, a consequence of the surjectivity of $\hat{\phi}$ provided by Proposition \ref{LQ} is the following result.

\begin{corollary} \label{gcd}
Let $(U,\phi)$ be an $n$-dimensional chart in a complete, metrically doubling Lipschitz differentiability space $(X,d,\mu)$, and suppose that $F \colon X \rightarrow \R^N$ is a bi-Lipschitz embedding. Then for $\mu$-almost every $x \in U$, the set $\Tan(F(X), F(x))$ consists of one element, canonically isometric to a fixed $n$-dimensional linear subspace of $\R^N$.
\end{corollary}

For a set $E\subset\R^N$ and a point $x\in E$, to say that each element of $\Tan(E,x)$ is \ti{canonically isometric} to $\hat{E}\subset\R^N$ means that, whenever $\lambda_j\rightarrow\infty$ and
$$ (\lambda_j(E-x), 0)$$
converges in the pointed Hausdorff sense in $\R^N$, the limit is $(\hat{E},0)$. This is stronger than saying that $\Tan(E,x)$ consists only of one element, and hence Cheeger \cite[Section 14]{Ch99} uses the phrase ``unique in the strong sense'' for this.

\begin{proof}[Proof of Corollary \ref{gcd}]
We give a brief sketch of the proof, whose details can be found in Theorem 14.1 of \cite{Ch99} or Corollary 8.1 of \cite{GCD15}.

Let $Y=F(X)\subset\R^N$. Consider a point $x\in X$ at which Proposition \ref{LQ} applies, and at which $F$ is differentiable with derivative $DF(x)\colon \R^n\rightarrow\R^N$. Let $y = F(x)$ and consider any tangent $(\hat{Y},\hat{y})\in \Tan(Y,y)$, which we can view as a pointed Hausdorff limit of rescalings of $Y-y$ in $\R^N$. By passing to subsequences, we may obtain an associated tangent $(\hat{X},\hat{x})\in \Tan(X,x)$ along with associated tangent mappings $\hat{\phi} \colon \hat{X}\rightarrow\R^n$ and $\hat{F} \colon \hat{X}\rightarrow \hat{Y}\subset\R^N$. Note that $\hat{F}$ is a bi-Lipschitz homeomorphism onto $\hat{Y}$. In addition, $\hat{\phi}$ is surjective by Proposition \ref{LQ}.

Differentiability implies that $\hat{F}$ factors as $DF(x) \circ \hat{\phi}$. Since $\hat{F}$ is bi-Lipschitz, $DF(x)$ must be injective. It follows from this and the surjectivity of $\hat{\phi}$ that 
$$\hat{Y} = F(\hat{X}) = DF(x)\circ\hat{\phi}(\hat{X}) = DF(x)(\R^n),$$ 
which is a fixed $n$-dimensional linear subspace of $\R^N$.
\end{proof}

Corollary \ref{gcd} is precisely the statement used in the Bonk--Kleiner proof of Theorem \ref{BKthm}. As discussed above, it is this result that we must generalize to Carnot group targets.

\subsection{Carnot groups}

Here we give some brief background on Carnot groups. For more, we refer the reader to \cite{Mo} or \cite[Section 2]{CC06}.

A Carnot group is a simply connected nilpotent Lie group $\G$ whose Lie algebra $\mathfrak{g}$ admits a stratification
$$\mathfrak{g} = V_1 \oplus \cdots \oplus V_s,$$
where the first layer $V_1$ generates the rest via $V_{i+1} = [V_1,V_i]$ for all $1 \leq i \leq s$, and we set $V_{s+1} = \{0\}$. The exponential map $\exp \colon \mathfrak{g} \rightarrow \G$ is a diffeomorphism, so choosing a basis for $\mathfrak{g}$ gives exponential coordinates for $\G$. Equipped with any left-invariant Riemannian metric, $\G$ is complete and the Lie exponential map coincides with the Riemannian exponential map. For each $x \in \G$, we will use $L_x \colon \G \rightarrow \G$ to denote the left multiplication map $y \mapsto xy$.

A natural family of automorphisms of $\G$ are the dilations $\delta_{\lambda} \colon \G \rightarrow \G$, for $\lambda >0$. On the Lie algebra level, these are defined by 
\begin{equation}\label{dilation}
v \mapsto \lambda^i v, \hspace{0.3cm} \text{for } v \in V_i,
\end{equation}
and one can see that this gives a Lie algebra isomorphism. Conjugating this back to $\G$ by the exponential map gives $\delta_{\lambda}$. 

On any Carnot group $\G$, there are metrics that interact nicely with the translations and dilations, in the sense that each $L_x$ is an isometry and $\delta_{\lambda}$ scales distances by the factor $\lambda$.

\begin{definition}\label{homogdistance}
A metric $d \colon \G \times \G\rightarrow \R_{\geq 0}$ is called a \textit{homogeneous distance} if it induces the manifold topology of $\G$, it is left-invariant, i.e.
$$ d(xy,xz) = d(y,z) \text{ for all } x,y,z \in \G,$$
and it is $1$-homogeneous with respect to the dilations $\delta_\lambda$ defined above:
$$ d(\delta_\lambda(x), \delta_\lambda(y)) = \lambda d(x,y) \text{ for all } \lambda>0 \text{ and } x,y\in\G.$$
\end{definition}

For example, given an inner product on the horizontal layer $V_1$, the associated (sub-Riemannian) Carnot--Caratheodory metric $d_{cc}$ is a homogeneous distance.

\begin{remark}\label{homogbilip}
It is a simple fact that, for any two homogeneous distances $d_1, d_2$ on $\G$, the identity map $(\G,d_1) \rightarrow (\G, d_2)$ is a bi-Lipschitz homeomorphism. In particular, a curve is rectifiable in $(\G,d_1)$ if and only if it is rectifiable in $(\G,d_2)$.
\end{remark}

Let $d$ be a homogeneous distance on $\G$, and let $\mu$ be the associated Hausdorff measure, so that $(\G, d, \mu)$ is a metric measure space. The left-invariance and 1-homogeneity of $d$ easily imply that $(\G, d, \mu)$ is Ahlfors regular. Another consequence is that every tangent to any point in $(\G, d)$ is isometric to $(\G, d, 0)$ itself, much like tangents of finite-dimensional Euclidean spaces. As we did for those, we will identify any $(\hat{\G},\hat{x}) \in \Tan(\G,x)$ with $(\G,0)$. Similarly, if $E$ is a closed subset of a Carnot group $\G$, then any tangent of $E$ can be viewed as a pointed Hausdorff limit of pointed rescalings of $E$ in $\G$ and identified with a pointed closed subset $(\hat{E},0)$ of $\G$. The terminology, used in $\R^n$ in Corollary \ref{gcd}, that tangents can be ``canonically isometric'' to subsets of $\G$, therefore also makes sense in Carnot groups. Namely, we say that each element of $\Tan(E,x)$ is canonically isometric to $\hat{E}\subset\G$ if, whenever $\lambda_j\rightarrow\infty$ and
$$ (\delta_{\lambda_j}(L_{-x}(E)), 0)$$
converges in the pointed Hausdorff sense, the limit is $(\hat{E},0)$. 

For the remainder of this section, we equip $\G$ with a fixed homogeneous distance $d$ and its corresponding Hausdorff measure. We can then talk about Lipschitz functions on $\G$ or Lipschitz mappings between Carnot groups. Note that the collection of such maps does not depend on the chosen homogeneous distance.

Lipschitz mappings between Carnot groups admit a form of differentiation more robust than that discussed in the previous subsection, one which takes into account the group structure.

\begin{theorem}[Pansu \cite{Pa89}]
Let $f \colon \G_1 \rightarrow \G_2$ be a Lipschitz map between Carnot groups. Then for almost every $x \in \G_1$, the sequence of maps
$$\delta_{\lambda} \circ (L_{f(x)^{-1}} \circ f \circ L_x) \circ \delta_{\lambda^{-1}}$$
converges uniformly on compact sets, as $\lambda \rightarrow \infty$, to a Lie group homomorphism $Df(x) \colon \G_1 \rightarrow \G_2$ that commutes with the dilations.
\end{theorem}

The following is an immediate consequence of Pansu's theorem, in the case that $\G_2 = \R$. Let $n = \dim(V_1)$ be the vector space dimension of the horizontal layer of $\mathfrak{g}$. There is a natural ``horizontal projection''
$$\pi \colon \G \rightarrow V_1 \simeq \R^n$$
obtained by composing $\exp^{-1}$ with the vector space projection $P \colon \mathfrak{g} \rightarrow V_1$. If $f \colon \G \rightarrow \R$ is a Lipschitz function, then for almost every $x \in \G$, the mapping $Df(x)$ factors as $Df(x) = A \circ \pi$, where $A \colon V_1 \rightarrow \R$ is linear.

The following lemma summarizes the additional basic properties of $\pi$ that we will need below.

\begin{lemma}\label{pifacts}
Let $\G$ be a Carnot group whose horizontal layer $V_1$ has dimension $n$, and let $\pi \colon \G \rightarrow V_1 \simeq \R^n$ be the associated horizontal projection.
\begin{enumerate}[(i)]
\item\label{piLD} $\G$ is a Lipschitz differentiability space with $n$-dimensional chart $(\G, \pi)$.
\item\label{pihomomorphism} $\pi$ is a group homomorphism.
\item\label{pidilation} $\pi$ commutes with dilations: $\pi(\delta_{\lambda}(x)) = \lambda \pi(x)$.
\item\label{pitangent} For all $x\in \G$, every element of $\Tan(\G, x, \pi)$ is isometric to $(\G, 0, \pi)$.
\item\label{piLQ} $\pi$ is a Lipschitz quotient map onto $V_1 \simeq \R^n$.
\item\label{picurve} If $\gamma \colon [0,1]\rightarrow \G$ is a non-constant Lipschitz curve, then $\pi \circ \gamma$ is non-constant.
\item\label{pilift} If $\gamma \colon [0,1]\rightarrow V_1$ is a Lipschitz curve and $x \in \pi^{-1}(\gamma(0))$, then there is a \textit{unique} Lipschitz curve $\tilde{\gamma}\colon [0,1]\rightarrow \G$ such that $\pi(\tilde{\gamma}(t)) = \gamma(t)$ and $\tilde{\gamma}(0) = x$. 
\end{enumerate}
\end{lemma}
\begin{proof}

That $\G$ is a Lipschitz differentiability space follows from the fact that all Carnot groups are PI spaces (\cite{Je}, \cite{HK98}) as well as Cheeger's theorem that all PI spaces are Lipschitz differentiability spaces \cite{Ch99}. That $\pi \colon \G\rightarrow\R^n$ serves as a global chart can be deduced from Pansu's theorem and the fact that for every group homomorphism $L\colon \G\rightarrow \R$ that commutes with dilations, there is a linear map $A \colon V_1 \rightarrow \R$ such that $L = A \circ \pi$. Here we use that $\pi$ is, in fact, a Lipschitz map.

The Baker--Campbell--Hausdorff formula shows that $\pi$ is a group homomorphism. Indeed,
$$\pi(xy) = P(\exp^{-1}(xy)) = P(\exp^{-1}(x) + \exp^{-1}(y) + v),$$
where $v \in [\mathfrak{g},\mathfrak{g}] = V_2 \oplus \cdots \oplus V_s$. Thus,
$$\pi(xy) = P(\exp^{-1}(x)) + P(\exp^{-1}(y)) = \pi(x) + \pi(y).$$

That $\pi$ commutes with the dilations on $\G$ and $V_1$ follows directly from the fact that dilations act on $V_1$ by simple scalings, as in \eqref{dilation}. Part \eqref{pitangent} then follows immediately from parts \eqref{pihomomorphism} and \eqref{pidilation}.

Part \eqref{piLQ} can be seen as follows. First of all, by Remark \ref{homogbilip}, it suffices to assume that $d=d_{cc}$. By translation and dilation invariance, it further suffices to show that $\pi(B(0,1))$ contains an open ball around $0$ in $V_1\simeq \R^n$. Each element $v_1 \in B_{V_1}(0,1)\subseteq V_1$ gives rise to an element $x \in B(0,1)$ by exponentiating $v_1 \oplus 0 \oplus ... \oplus 0$. Since $\pi(x)=v_1$, we see that $\pi(B(0,1)) \supseteq B_{V_1}(0,1)$.

Part \eqref{picurve} is immediate from the definition of the Carnot--Carath\'eodory metric and Remark \ref{homogbilip}.

Finally, Property \eqref{pilift} follows from \cite[Proposition 2.3]{TY13}, using that $\gamma$ is absolutely continuous. To show that the lift $\tilde{\gamma}$ is Lipschitz, and not just absolutely continuous, one should use the fact that the horizontal derivative of $\tilde{\gamma}(t)$ coincides almost everywhere with the derivative of $\gamma(t)$, and so is essentially bounded.
\end{proof}

We saw in Corollary \ref{gcd} that tangents of Lipschitz differentiability spaces embedded in Euclidean space are themselves just Euclidean subspaces. For the appropriate generalization to Lipschitz differentiability spaces embedded in Carnot groups, we need the correct notion of a Carnot subgroup. This is given by the following definition.

\begin{definition}\label{carnotsubgroup}
Let $\G$ be a Carnot group with Lie algebra $\mathfrak{g}$ and horizontal layer $V_1\subseteq \mathfrak{g}$. Let $V\subset V_1$ be a vector subspace, and let $\mathfrak{h} \subset \mathfrak{g}$ be the stratified Lie sub-algebra generated by $V$. The homogeneous subgroup $H = \exp(\mathfrak{h}) \subset \G$ is called the \textit{Carnot subgroup generated by $V$}.
\end{definition}

Note that $H$ is itself a Carnot group, and any homogeneous metric $d$ on $\G$ restricts to a homogeneous metric on $H$. Moreover, $H$ is rectifiably connected in this restricted metric. 
When $d=d_{cc}$ is Carnot--Caratheodory, one can make the stronger statement that $H$ is a totally geodesic subgroup of $\G$.

\section{Symmetric spaces and Gromov hyperbolic groups}\label{symgromov}

Let $X$ be a non-compact rank-one symmetric space, so that $X = \Hy_\R^n$, $\Hy^n_\C$, $\Hy^n_\Q$, or $\Hy_\O^2$ for some $n \geq 2$, where $\Q$ and $\O$ denote the quaternion and octonion division algebras. As a convention, we normalize the Riemannian metric on $X$ to have maximal sectional curvature equal to $-1$. With the induced length metric $d_X$, the metric space $(X,d_X)$ is therefore $\CAT(-1)$. We remind the reader that $X$ is homogeneous and isotropic, so $\Isom(X)$ acts transitively on the unit tangent bundle of $X$.

For $\Gamma$ a discrete group, an isometric and properly discontinuous action $\Gamma \acts X$ is said to be \textit{convex-cocompact} if there is a convex, $\Gamma$-invariant subset $C(\Gamma) \subset X$ with $C(\Gamma)/\Gamma$ compact. This is equivalent to the seemingly weaker property that, for any point $p \in X$, the orbit map $g \mapsto g(p)$ gives a quasi-isometric embedding of $\Gamma$ into $X$ (cf. \cite[Section 1.8]{Bou95} and \cite[Section 3]{Gui}). As $C(\Gamma)$ is Gromov hyperbolic, the \v{S}varc--Milnor lemma implies that $\Gamma$ is necessarily a finitely-generated hyperbolic group. From now on, we will always assume that $\Gamma$ is non-elementary, i.e., is not finite and not virtually cyclic.

\subsection{Visual boundaries of rank-one symmetric spaces}\label{visualsec}

Let $\partial X$ denote the visual boundary of $X$, which is a topological sphere of dimension $\dim_\R(X)-1$. There are two natural classes of metrics on this boundary on which we focus. First, the visual metrics on $\partial X$ are defined by
$$d_p(x,y) = e^{-(x,y)_p} \hspace{0.5cm} \text{for } x,y  \in \partial X$$
for any $p \in X$, where $(x,y)_p$ is the Gromov product of $x$ and $y$ based at $p$ \cite[Section 2.5]{Bou95}. Every element of $\Isom(X)$ extends to a homeomorphism of $\partial X$ that is M\"obius with respect to any visual metric.

The parabolic visual metrics are similar but are better suited to the parabolic models for $X$. Namely, for any $\omega \in \partial X$ and $q \in X$, define
$$d_{\omega, q}(x,y) = e^{-(x,y)_{\omega,q}}, \hspace{0.5cm} \text{for } x,y \in \partial X \backslash \{\omega\}$$
where 
$$(x,y)_{\omega,q} = \lim_{p \rightarrow \omega}\lp (x,y)_p - d_X(p,q) \rp$$ 
is a limit taken along a geodesic ray in $X$ that is asymptotic to $\omega$. This is a metric on $\partial X \backslash \{\omega\}$, which is obtained as a limit of rescaled visual metrics $d_p$, as $p$ tends toward $\omega$ non-tangentially. We refer the reader to \cite[Section 2]{BKqf} for details, noting that the definition we give for $(x,y)_{\omega,q}$ is a consequence of \cite[Lemma 2.1]{BKqf}.

Each visual metric is M\"obius equivalent to each parabolic visual metric on their common domains \cite[Lemma 2.3]{BKqf}. An immediate consequence is that any two parabolic visual metrics are M\"obius equivalent on their common domains. Moreover, for fixed $\omega \in \partial X$, the metrics $d_{\omega,q}$ and $d_{\omega, q'}$ differ by a scalar multiple for any $q,q' \in X$. This follows from the fact that $\lim_{p \rightarrow \omega} \lp d_X(p,q) - d_X(p,q')\rp$ exists. These two metrics coincide precisely when $q$ and $q'$ lie on the same horosphere based at $\omega$.

We should note that the above discussion holds equally well for boundaries of $\CAT(-1)$ metric spaces. The important point for us is that, when $X$ is non-compact rank-one symmetric, the boundary has much additional structure. Namely, given a point $\omega \in \partial X$, there is a natural identification of $\partial X \backslash \{\omega\}$ with a Carnot group $\G$. Here, the horizontal distribution on $\G$ arises from vectors that are tangent to the lines and circles in $\partial X \backslash \{\omega\}$ formed by isometric copies of $\Hy_\R^2$ in $X$. Moreover, the subgroup of $\Isom(X)$ that fixes $\omega$ corresponds to the collection of affine transformations of $\G$. In particular, this includes all left translations and dilations.

\begin{lemma} \label{paralemma}
Each parabolic visual metric $d_{\omega,q}$ is a homogeneous distance on $\G = \partial X \backslash \{\omega\}$.
\end{lemma}

\begin{proof}
That $d_{\omega,q}$ induces the Euclidean topology is a direct consequence of the standard fact that any visual metric $d_p$ induces the spherical topology on $\partial X$. The other two properties are consequences of the identity 
$$d_{\omega, g(q)}(g(x),g(y)) = d_{\omega, q}(x,y) \hspace{0.3cm} \text{for } x,y \in \partial X \backslash \{\omega\}$$
whenever $g \in \Isom(X)$ fixes $\omega$. 

Indeed, every left translation of $\partial X \backslash \{\omega\}$ is the boundary map of some element $g \in \Isom(X)$ that fixes $\omega$ and preserves the horospheres in $X$ that are based at $\omega$. As $g(q)$ and $q$ lie on the same horosphere, we have $d_{\omega, g(q)} = d_{\omega, q}$, which shows that $d_{\omega, q}$ is left-invariant.

Similarly, every dilation $\delta_{\lambda}$ of $\partial X \backslash \{\omega\}$ is the boundary map of some element $g \in \Isom(X)$ that acts as a translation by distance $\log\lambda$ along a geodesic in $X$ that is asymptotic to $\omega$. This means that 
$$\lim_{p \rightarrow \omega} \lp d_X(p,q) - d_X(p,g(q))\rp = \log\lambda,$$
so we obtain $d_{\omega, g(q)} = \lambda d_{\omega, q}$. This shows that $d_{\omega, q}$ is 1-homogeneous.
\end{proof}

\begin{remark}
Identifying $\partial X \backslash \{\omega\}$ with a horosphere in $X$ based at $\omega$, one can obtain a sub-Riemannian Carnot--Caratheodory metric on $\partial X \backslash \{\omega\}$ as a limit of Riemannian metrics (cf. \cite{BuKu}). In general, the parabolic visual metrics $d_{\omega,q}$ are not geodesic, and so will not coincide with the sub-Riemannian metric. However, the parabolic visual metrics are better suited for our work in Section \ref{mainproof}.
\end{remark}

We should note that the more general identity 
$$d_{g(\omega), g(q)}(g(x),g(y)) = d_{\omega, q}(x,y) \hspace{0.3cm} \text{for } x,y \in \partial X \backslash \{\omega\}$$
holds for any $g \in \Isom(X)$. Using that $\Isom(X)$ acts transitively on the unit tangent bundle of $X$, for any two pairs $\omega, \omega' \in \partial X$ and $q,q' \in X$, there is $g \in \Isom(X)$ with $g(\omega) = \omega'$ and $g(q) = q'$. Thus, the parabolic boundaries $(\partial X \backslash \{\omega\}, d_{\omega,q})$ and $(\partial X \backslash \{\omega'\}, d_{\omega',q'})$ are isometrically equivalent.

We therefore consider the Carnot group $\G$, equipped with any parabolic metric $d = d_{\omega,q}$, to be a model for the local geometry of $(\partial X, d_p)$. By our discussion above, it is clear that $(\G,d)$ is M\"obius equivalent to $(\partial X \backslash \{\omega\},d_p)$, regardless of the choice of $\omega \in \partial X$. In particular, this means that $(\partial X,d_p)$ is locally bi-Lipschitz equivalent to $(\G,d)$, so it is also locally bi-Lipschitz equivalent to $\G$ equipped with any homogeneous distance. 

Of course, the Carnot groups that locally model boundaries of non-compact rank-one symmetric spaces are a special sort. They are either Euclidean, Heisenberg, or of ``Heisenberg type":
\begin{enumerate}
\item[(i)] if $X = \H_{\R}^n$, then $\G = \R^{n-1}$ is Euclidean space;
\item[(ii)] if $X = \H_{\C}^n$, then $\G = \mathcal{H}_{\C}^{n-1}$ is the $n$-th Heisenberg group;
\item[(iii)] if $X = \H_{\Q}^n$, then $\G = \mathcal{H}_{\Q}^{n-1}$ is the $n$-th quaternionic Heisenberg group;
\item[(iv)] if $X = \H_{\O}^2$, then $\G = \mathcal{H}_{\O}^1$ is the first octonionic Heisenberg group.
\end{enumerate}
Together, these Carnot groups form the class of \textit{Iwasawa groups}. In each case, isometries of $X$ act on $\partial X$ by conformal maps, i.e. smooth maps for which the restriction of the derivative to the horizontal layer is a similarity. For $\Hy_\R^n$ this is classical, and the boundary action is by (classical) M\"obius transformations; for $\Hy_\C^n$, this is shown in \cite[p. 328]{KR2}; the other cases follow from \cite[Corollary 11.2]{Pa89} and \cite[Corollary 7.2]{CC06}.

If $\Gamma \acts X$ is a convex-cocompact action with $\Gamma$ non-elementary, the limit set $\Lambda(\Gamma) \subset \partial X$ is defined to be the visual boundary of any convex, $\Gamma$-invariant subset $C(\Gamma) \subset X$ for which $C(\Gamma)/\Gamma$ is compact. Equivalently, if $p \in X$ is any point, $\Lambda(\Gamma)$ is the image of the boundary map $\partial \Gamma \rightarrow \partial X$ induced by the quasi-isometric orbit embedding $\Gamma \rightarrow X$. In particular, $\Lambda(\Gamma)$ is a closed subset, and the M\"obius action $\Gamma \acts \partial X$ leaves $\Lambda(\Gamma)$ invariant. Thus, we obtain a natural M\"obius action $\Gamma \acts \Lambda(\Gamma)$, when we equip $\Lambda(\Gamma)$ with the restriction of any visual metric on $\partial X$. It is known that the corresponding Hausdorff measure is Ahlfors regular \cite[Section 2.7]{Bou95}.

\subsection{Boundaries of Gromov hyperbolic groups} \label{consec}

Much of the content of this subsection is background and will not be needed in the remainder of the paper. However, we believe it should be recorded in the literature, and some of it is needed in the proof of our main result. For further background and terminology about Gromov hyperbolic groups, we refer the reader to \cite{BKconf}.

Let $\Gamma$ be a non-elementary Gromov hyperbolic group, by which we mean in particular that $\Gamma$ is finitely-generated. The visual boundary $\partial \Gamma$ is perfect and compact, and it admits a collection of visual Gromov metrics, each of which is Ahlfors regular \cite{Coor}. Any two such metrics are quasisymmetrically equivalent. Moreover, the action of $\Gamma$ on itself by left multiplication extends naturally to a boundary action $\Gamma \acts \partial \Gamma$ that is uniformly quasi-M\"obius (with linear distortion function $\eta(t)$) with respect to any visual Gromov metric \cite[Section 6]{Kin}.

We will use $\mathcal{J}_{\AR}(\partial \Gamma)$ to denote the Ahlfors regular conformal gauge of $\partial \Gamma$ that contains these metrics, i.e., the collection of all Ahlfors regular metric spaces quasisymmetric to $\partial \Gamma$. By definition,
$$\cdim_{\AR}(\partial \Gamma) = \inf \{\Hdim(Z) : Z \in \mathcal{J}_{\AR}(\partial \Gamma) \}.$$
Many quasi-isometric uniformization statements about $\Gamma$ boil down to finding a highly regular metric in $\mathcal{J}_{\AR}(\partial \Gamma)$. The following gives a list of equivalent notions for ``highly regular." In what follows, all of the metric spaces that appear will be Ahlfors regular, and we endow them with the corresponding Hausdorff measure.

The following result is surely known to experts, but we include it as a useful summation. The definitions of Poincar\'e inequalities and Loewner spaces can be found in \cite{HK98} or \cite{He01}.

\begin{theorem} \label{tfae}
For $Z \in \mathcal{J}_{\AR}(\partial \Gamma)$ of Hausdorff dimension $Q>1$, the following are equivalent.
\begin{enumerate}[(i)]
\item\label{tfaePI} $Z$ admits a $(1,p)$-Poincar\'e inequality for some $p \geq 1$.
\item\label{tfaeLD} $Z$ is a Lipschitz differentiability space.
\item\label{tfaeCD} $Z$ has Ahlfors regular conformal dimension equal to $Q$.
\item\label{tfaeQL} $Z$ is a $Q$-Loewner space.
\item\label{tfaePM} $Z$ admits a path family of positive $p$-modulus for some $p \geq 1$.
\end{enumerate}
\end{theorem}

\begin{proof}
As $Z$ is in $\mathcal{J}_{\AR}(\partial \Gamma)$ and has Hausdorff dimension $Q$, it is Ahlfors $Q$-regular. We first show that properties \eqref{tfaePI} through \eqref{tfaeQL} are equivalent.

The implication \eqref{tfaePI} implies \eqref{tfaeLD} is a consequence of the main theorem of Cheeger in \cite{Ch99}, using the fact that $Z$ is Ahlfors regular and hence a doubling metric measure space. 

To see that \eqref{tfaeLD} implies \eqref{tfaeCD}, we first note that \cite[Theorem 1.15]{CKS15} shows that $Z$ has a tangent $\hat{Z}$ that admits an \ti{Alberti representation} supported on geodesic lines. (Here, $\hat{Z}$ is also Ahlfors $Q$-regular, so we use the corresponding Hausdorff measure.)  We refer the reader to \cite{CKS15} for the definition of an Alberti representation; all we will need is the fact that the geodesic lines in the support of this Alberti representation constitute a path family in $\hat{Z}$ that has positive $1$-modulus. Restricting these geodesics to their intersections with a large fixed ball in $\hat{Z}$ gives a path family of positive $1$-modulus inside of a compact set. By H\"older's inequality, this path family has positive $Q$-modulus as well. It now follows from \cite[Proposition 4.1.8]{MTcdim}, and the fact that $Z$ itself is Ahlfors $Q$-regular, that $Z$ has Ahlfors regular conformal dimension equal to $Q$.

That \eqref{tfaeCD} implies \eqref{tfaeQL} is a consequence of \cite[Theorem 1.3]{BKconf} (and the remark following the statement of that theorem).

Since $Z$ is Ahlfors $Q$-regular, property \eqref{tfaeQL}, that $Z$ is $Q$-Loewner, implies that $Z$ has a $(1,Q)$-Poincar\'e inequality and hence property \eqref{tfaePI} \cite[Theorem 5.12]{HK98}.

It remains to show that \eqref{tfaePM} is equivalent to the other properties. Of course, property \eqref{tfaeQL} implies the existence of a path of positive $Q$-modulus and hence property \eqref{tfaePM}. On the other hand, if $Z$ admits a path family with positive $p$-modulus for some $p\geq 1$, then \cite[Theorem 4.0.5]{KL} shows that $Z$ has a tangent with a path family of positive $1$-modulus, hence of positive $Q$-modulus as above. It then follows from \cite[Corollary 6.1.8]{MTcdim} that $Z$ has Ahlfors regular conformal dimension equal to $Q$, i.e., that \eqref{tfaeCD} holds.
\end{proof}

\begin{remark} \label{tfaermk}
When $\Gamma \acts X$ is a convex-cocompact action on a non-compact rank-one symmetric space $X$, as in the previous subsection, the boundary homeomorphism between $\partial \Gamma$ and $\Lambda(\Gamma)$ is quasisymmetric. As $\Lambda(\Gamma)$ is Ahlfors regular, we have $\Lambda(\Gamma) \in \mathcal{J}_{\AR}(\partial \Gamma)$. Consequently,
$$\Hdim(\Lambda(\Gamma)) \geq \cdim_{\AR}(\Lambda(\Gamma)) = \cdim_{\AR}(\partial \Gamma),$$ which is the inequality that appears in Theorems \ref{BKthm} and \ref{mainthm}.

The case of equality $\Hdim(\Lambda(\Gamma)) = \cdim_{\AR}(\partial \Gamma)$ means precisely that condition \eqref{tfaeCD} in Theorem \ref{tfae} holds for $Z = \Lambda(\Gamma)$ and $Q = \cdim_{\AR}(\partial \Gamma)$. Thus, if $Q > 1$, then the other conditions hold for $\Lambda(\Gamma)$ as well. 

On the other hand, if $Q \leq 1$, there is only one possibility for the action $\Gamma \acts X$. Indeed, since $Q \leq 1$, the topological dimension of $\Lambda(\Gamma)$ is either $0$ or $1$. If the former, then $\partial \Gamma$ also has topological dimension 0, and $\Gamma$ is virtually a free group by \cite[Theorem 8.1]{KapBen}. This means that $\partial \Gamma$ is a uniformly perfect Cantor set, which is known to have Ahlfors regular conformal dimension equal to 0. Hence, $Q=0$, which implies that $\Lambda(\Gamma)$ is finite, and so $\Gamma$ is elementary, a contradiction. Thus, the topological dimension of $\Lambda(\Gamma)$ must be equal to 1, which then means that $Q=1$ as well. Applying \cite[Theorem 1.1]{BKqf} shows that $\Gamma$ is virtually Fuchsian and the action $\Gamma \acts X$ stabilizes an isometric copy of $\Hy^2_\R$ in $X$.

These arguments justify our subsequent restriction to the case $Q>1$ and our phrasing of Theorem \ref{BKthm} above (which, in \cite{BKconf}, is stated only for $Q > 1$).
\end{remark}

In the remainder of the section, we make some general remarks about the types of Lipschitz differentiability structures that can appear on boundaries of hyperbolic groups. 

Let $\Gamma$ be a Gromov hyperbolic group with $\cdim_{\AR}(\partial \Gamma) = Q >1$. If there is $Z \in \mathcal{J}_{\AR}(\partial \Gamma)$ that is a Lipschitz differentiability space, then Theorem \ref{tfae} guarantees that $\dim_H(Z) = Q$. In other words, the Hausdorff dimension of any highly regular metric on $\partial \Gamma$ depends only on the quasi-isometry class of $\Gamma$. Being slightly imprecise, one could simply refer to $Q$ as the Hausdorff dimension of $\partial \Gamma$. 

We claim that a similar statement holds for the dimension of the differentiability structure. In fact, we can establish something stronger.

\begin{lemma} \label{ergodic}
\begin{enumerate}[(i)]
\item If $Z \in \mathcal{J}_{\AR}(\partial \Gamma)$ is a Lipschitz differentiability space, then the action $\Gamma \acts Z$, obtained by conjugating the boundary action $\Gamma \acts \partial \Gamma$ by a quasisymmetric homeomorphism between $\partial \Gamma$ and $Z$, is ergodic with respect to $Q$-dimensional Hausdorff measure.
\item If $Z_1, Z_2 \in \mathcal{J}_{\AR}(\partial \Gamma)$ are Lipschitz differentiability spaces, then for any Borel sets $U_1 \subset Z_1$, $U_2 \subset Z_2$ of positive measure, there are positive measure subsets $A_1 \subset U_1$, $A_2 \subset U_2$ that are bi-Lipschitz equivalent.
\end{enumerate}
\end{lemma}

\begin{proof}
To verify (i), first note that $Z$ is Ahlfors $Q$-regular and $Q$-Loewner by Theorem \ref{tfae}. Let $\mu$ denote the $Q$-dimensional Hausdorff measure, which we may normalize to have $\mu(Z) = 1$.

The action $\Gamma \acts Z$ is uniformly quasi-M\"obius and, hence, uniformly quasiconformal: there is $K<\infty$ for which each $g \in \Gamma$ acts as a $K$-quasiconformal homeomorphism of $Z$. Using that $Z$ is Ahlfors $Q$-regular and $Q$-Loewner, this means that each $g \in \Gamma$ is absolutely continuous in measure \cite[Corollary 8.15]{HKST}. Hence, $\Gamma \acts Z$ is a measure-class preserving action. Ergodicity of such an action means, as usual, that any $\Gamma$-invariant Borel set has measure 0 or 1.

To show this, we note that each $g \in \Gamma$ has a uniform density property (in the sense of \cite[Section 6]{KMS12}) with uniform distortion control. Actually, we will use a slightly different property, but which is of the same spirit: there is a homeomorphism $\phi \colon [0,\infty) \rightarrow [0,\infty)$ such that
$$\mu(g(E \cap B)) \leq \phi \lp \frac{\mu(E \cap B)}{\mu(B)} \rp$$
for every $g \in \Gamma$, every Borel set $E \subset Z$, and every ball $B \subset Z$. A minor modification of the proof of \cite[Theorem 6.3]{KMS12} shows this easily, once we remark that $Z$ satisfies a $(1,p)$-Poincar\'e inequality for some $1< p < Q$ by the main result of \cite{KZ08}.

Now, suppose that $E \subset Z$ is a $\Gamma$-invariant Borel set with $\mu(E) < 1$. Let $\e > 0$ be arbitrary, and let $B \subset Z$ be a ball, centered at a point of density for $Z\backslash E$, with radius small enough that
$$\phi \lp \frac{\mu(E \cap B)}{\mu(B)} \rp < \frac{\e}{2}.$$
By \cite[Lemma 5.1]{BKqm}, there are elements $g \in \Gamma$ for which $\diam(Z \backslash g(B))$ is arbitrarily small. In particular, we can find $g \in \Gamma$ for which
$$\mu(Z \backslash g(B)) < \e/2.$$
Using that $E$ is $\Gamma$-invariant, we have $E = g(E) \subset g(E \cap B) \cup \lp Z \backslash g(B) \rp$,
and thus
$$\mu(E) \leq \mu(g(E \cap B)) + \mu(Z \backslash g(B)) < \phi \lp \frac{\mu(E \cap B)}{\mu(B)} \rp + \frac{\e}{2} < \e.$$
As $\e>0$ was arbitrary, we see that $\mu(E) = 0$. Hence, $\Gamma \acts Z$ is ergodic.

Let us now verify part (ii). Again, we know that $Z_1$ and $Z_2$ are both Ahlfors $Q$-regular and $Q$-Loewner. Moreover, there is a quasisymmetric homeomorphism $f \colon Z_1 \rightarrow Z_2$, which is absolutely continuous in measure by \cite[Corollary 8.15]{HKST}. In particular, the image $f(U_1)$ has positive measure in $Z_2$. By part (i), the action $\Gamma \acts Z_2$ is ergodic, so there is $g \in \Gamma$ for which $g(f(U_1)) \cap U_2$ has positive measure. Replacing $f$ by the composition $g \circ f$, we may suppose without loss of generality that $f(U_1) \cap U_2$ has positive measure. This also means that $U_1 \cap f^{-1}(U_2)$ has positive measure.

It is shown in \cite[Section 10]{HKST} that if $f \colon Z_1 \rightarrow Z_2$ is a quasisymmetric homeomorphism between Ahlfors $Q$-regular, $Q$-Loewner spaces, then $Z_1$ can be covered, up to a set of measure zero, by Borel sets on which $f$ is Lipschitz. In particular, given any positive measure subset $B \subset Z_1$, there is a positive measure set $A \subset B$ on which $f$ is Lipschitz. Using the same argument for $f^{-1}$, along with the fact that $f$ and $f^{-1}$ preserve sets of measure zero, it is easy to see that we can actually find a positive measure subset $A \subset B$ on which $f$ is bi-Lipschitz. Applying this fact to the set $U_1 \cap f^{-1}(U_2) \subset Z_1$, we find a positive measure subset $A_1 \subset U_1\cap f^{-1}(U_2)$ on which $f$ is bi-Lipschitz. Then set $A_2 = f(A_1) \subset U_2$.
\end{proof}

An immediate consequence of Lemma \ref{ergodic}(ii) is that there is a unique integer $k$ for which each chart in any Lipschitz differentiability space $Z \in \mathcal{J}_{\AR}(\partial \Gamma)$ has dimension $k$. It makes sense to call this integer the analytic dimension of $\partial \Gamma$, and once again it depends only on the quasi-isometry class of $\Gamma$. Thus, in the case that $\mathcal{J}_{\AR}(\partial \Gamma)$ contains a Lipschitz differentiability space, there are three natural well-defined dimensions to consider: the topological dimension of $\partial \Gamma$, the Hausdorff dimension of $\partial \Gamma$, and the analytic dimension of $\partial \Gamma$. 

At the same time, it should be quite rare for $\mathcal{J}_{\AR}(\partial \Gamma)$ to contain a Lipschitz differentiability space. Currently, the only known examples come from uniform lattices in non-compact rank-one symmetric spaces and uniform lattices in certain types of hyperbolic buildings \cite{BP99}. It makes more sense to look for rigidity phenomena for such groups $\Gamma$. One might expect the topological, Hausdorff, and analytic dimensions of $\partial \Gamma$ to function as characteristic quantities for rigidity.

There are some intimations toward this type of rigidity when $\Gamma$ is a manifold group. If $\Gamma = \pi_1(M)$ with $M$ a closed, negatively curved Riemannian manifold, of dimension at least 3, then the action by deck transformations $\Gamma \acts \tilde{M}$ on the universal Riemannian cover has compact quotient. Rescaling the metric, we may assume that the maximal sectional curvature of $\tilde{M}$ equals $-1$. The visual boundary $\partial \tilde{M}$, equipped with any visual metric, is then in $\mathcal{J}_{\AR}(\partial \Gamma)$. 

If $\partial \tilde{M}$ has a path family of positive $p$-modulus for some $p \geq 1$, then a theorem of Connell shows that $\tilde{M}$ is isometric to a non-compact rank-one symmetric space (\cite[Theorem 4.3]{Con} for Patterson--Sullivan measures). This symmetric space is uniquely determined by the topological and Hausdorff dimensions of $\partial \tilde{M}$, as it is locally modeled by an Iwasawa group. It is not clear, however, whether the same conclusion holds if one assumes only that $\mathcal{J}_{\AR}(\partial \Gamma)$ contains a Lipschitz differentiability space.

Finally, let us record an analog of the Patterson--Sullivan case in Connell's result that uses some of the ideas we will see below. 

\begin{prop}\label{connell}
Let $\Gamma \acts X$ be a convex-cocompact action on a $\CAT(-1)$ space $X$ with $\Hdim(\Lambda(\Gamma)) > 1$. Suppose that $\Lambda(\Gamma)$ admits a path family of positive $p$-modulus for some $p \geq 1$. Then every geodesic in $X$ whose endpoints lie in $\Lambda(\Gamma)$ is contained in an isometrically embedded copy of $\Hy_\R^2$ in $X$ for which the boundary circle lies entirely in $\Lambda(\Gamma)$.
\end{prop}

In the setting of Connell's theorem, the action of $\Gamma$ on $X = \tilde{M}$ is cocompact, so the conclusion of Proposition \ref{connell} holds for all geodesics in $\tilde{M}$. This means that $\tilde{M}$ has hyperbolic rank at least 1, and a powerful theorem of Hamenst\"adt then implies that $\tilde{M}$ is symmetric \cite{Ha92}.

\begin{proof}
By a theorem of Bourdon \cite[Theorem 0.1]{Bou}, it suffices to show that any two points in $\Lambda(\Gamma)$ can be joined by a M\"obius circle that lies in $\Lambda(\Gamma)$. As M\"obius circles in $\partial X$ are closed under non-trivial limits, it actually suffices to show that for a dense set of points $\omega \in \Lambda(\Gamma)$, for every $\eta \in \Lambda(\Gamma) \backslash \{\omega\}$, there is a M\"obius circle in $\Lambda(\Gamma)$ containing $\omega$ and $\eta$. Recall, though, that the boundary action $\Gamma \acts \Lambda(\Gamma)$ is M\"obius, and the orbit of every point is dense. Thus, it suffices to prove the previous statement for a single point $\omega \in \Lambda(\Gamma)$. A natural rephrasing of this statement is that every $\eta \in \Lambda(\Gamma) \backslash \{\omega\}$ is contained in an isometric copy of $\R$ in the parabolic limit set $(\Lambda(\Gamma) \backslash \{\omega\}, d_{\omega,q})$ for some $q \in X$. Let us prove this formulation.

By Theorem \ref{tfae}, we know that $\Lambda(\Gamma)$ is a Lipschitz differentiability space. A result of Cheeger--Kleiner--Schioppa, Theorem \ref{LQ} above, guarantees that $\Lambda(\Gamma)$ has a tangent $Y$ for which there is a metric submersion $Y \rightarrow (\R^k,||\cdot||)$ onto a normed space. A theorem of Bonk--Kleiner \cite[Proposition 3.1]{BKqf} shows that there are points $\omega \in \Lambda(\Gamma)$ and $q \in X$ for which $Y$ is isometrically equivalent to $(\Lambda(\Gamma) \backslash \{\omega\}, d_{\omega,q})$. Thus, there is a metric submersion 
$$f \colon (\Lambda(\Gamma) \backslash \{\omega\}, d_{\omega,q}) \rightarrow (\R^k,||\cdot||).$$
For $\eta \in \Lambda(\Gamma) \backslash \{\omega\}$, let $\ell$ be a geodesic line in $\R^k$ through $f(\eta)$, e.g., the line in the first coordinate direction. By Lemma \ref{pathlifting} below, there is a lift of $\ell$ to a geodesic line in $\Lambda(\Gamma) \backslash \{\omega\}$ that contains $\eta$, as desired.
\end{proof}

It is desirable to understand better the global geometry of $X$ that can arise in this setting, even in the case that $\Gamma \acts X$ is cocompact. One can consider Theorem \ref{mainthm} to be a description of what happens when the ambient space $X$ is not only $\CAT(-1)$ but is in fact symmetric.

\section{Tangents of Lipschitz differentiability spaces in Carnot groups}\label{carnottangent}

The main result of this section is Theorem \ref{tan}, which essentially shows that if a Lipschitz differentiability space lies inside a Carnot group, then its tangents are Carnot subgroups. Theorem \ref{tan} is an important piece in the proof of Theorem \ref{mainthm}, but also has other interesting non-embedding consequences for Lipschitz differentiability spaces, as explained in subsection \ref{nonembedding}.

Let $\G$ be a Carnot group, and let $d$ be a homogeneous distance on $\G$. Let $X \subset \G$ be a subset for which there is a Radon measure $\mu$ on $X$ for which $(X,d,\mu)$ is a Lipschitz differentiability space. Recall the definition of Carnot subgroups given in Definition \ref{carnotsubgroup}. Note that $X$ is automatically metrically doubling, as a subset of the metrically doubling space $\G$.

\begin{theorem} \label{tan}
Let $(U,\phi)$ be a $k$-dimensional chart for $X$, and assume that for almost every $x \in U$, each $(\hat{X},\hat{x}) \in \Tan(X,x)$ is rectifiably connected. Then for almost every $x \in U$, the set $\Tan(X,x)$ consists of one element, which is canonically isometric to a Carnot subgroup of $\G$ generated by a fixed $k$-dimensional subspace of the horizontal layer of $\G$.
\end{theorem}

\begin{remark} \label{tanrmk}
If $(X,d,\mu)$ is a PI space (i.e., is doubling and satisfies a Poincar\'e inequality in the sense of \cite{HK98}), then the assumption that each $(\hat{X},\hat{x})\in \Tan(X,x)$ is rectifiably connected can be omitted in Theorem \ref{tan}. This is because  each element of $\Tan(X,x)$ will be quasiconvex and hence rectifiably connected.

The same is true if $(X,d,\mu)$ is a RNP Lipschitz differentiability space, in the sense of \cite{BL15}. This follows from Corollary 9.4 of \cite{BL15}.
\end{remark}

The following path lifting lemma for Lipschitz quotients is taken from \cite[Lemma 4.4]{BJLPS99} (equivalently \cite[Lemma 2.2]{JLPS00}). Although stated there only for Euclidean domains, the proof works equally well in the context below. Recall that a metric space is \ti{proper} if each closed ball in the space is compact. Every complete, doubling metric space is proper.

\begin{lemma}\label{pathlifting}
Let $X$ be a proper metric space and $Y$ a metric space. Let $F \colon X\rightarrow Y$ be a Lipschitz quotient with co-Lipschitz constant $c>0$, and let $\gamma \colon [0,T] \rightarrow Y$ be a $1$-Lipschitz curve with $\gamma(0)=F(x)$. Then there is a $(1/c)$-Lipschitz curve $\tilde{\gamma} \colon [0,T] \rightarrow X$ such that $\tilde{\gamma}(0)=x$ and $F\circ \tilde{\gamma} = \gamma$. 
\end{lemma}

We now let $\pi \colon \G \rightarrow V_1 \simeq \R^n$ be the global differentiability chart for $\G$, and let $\pi = (\pi_1,\ldots,\pi_n)$ denote its coordinates. Recall that $(U,\phi \colon X\rightarrow\R^k)$ is a differentiability chart for the space $(X,d,\mu)$ contained in $\G$.

\begin{lemma}\label{picharts}
The set $U$ can be covered by a finite number of charts, with chart maps of the form 
$$(\pi_{i_1},\ldots, \pi_{i_k}) \colon X \rightarrow \R^k$$
for some choice of $k$ distinct indices $i_1,\ldots,i_k$.
\end{lemma}

\begin{proof}
Let $\iota \colon X \hookrightarrow \G$ be the inclusion map, which is an isometry. Post-composing with $\pi \colon \G \rightarrow \R^n$ gives a Lipschitz map from $X$ to $\R^n$, which can be differentiated with respect to the chart $\phi \colon X \rightarrow \R^k$ for $U$. Thus, for almost every $x \in U$, there is a unique linear map
$$D\iota(x) \colon \R^k \rightarrow \R^n$$
for which
\begin{equation}\label{piderivative}
\pi(y)-\pi(x) = D\iota(x)(\phi(y) - \phi(x)) + o(d(x,y)), \hspace{0.3cm} y \in X.
\end{equation}
For almost every $x \in U$, if $(\hat{X},\hat{x}) \in \Tan(X,x)$, then the corresponding tangent maps $\hat{\phi} \colon \hat{X} \rightarrow \R^k$ and $\hat{\iota} \colon \hat{X} \rightarrow \G$ satisfy
$$\pi \circ \hat{\iota} = D\iota(x) \circ \hat{\phi}.$$
Here, we consider $\hat{X}$ as a subset of $\G$, so that $\hat{\iota}$ is just the inclusion map again. Moreover, $\hat{\phi} \colon \hat{X} \rightarrow \R^k$ is a Lipschitz quotient mapping.

Consider any non-zero vector $v\in\R^k$. By Lemma \ref{pathlifting}, there is a Lipschitz curve $\ell \colon [0,1] \rightarrow \hat{X} \subset \G$ with
$$\hat{\phi} \circ \ell(t) = tv.$$
In particular, we have
$$\pi \circ \ell (t) = D\iota(x)(tv) = t \cdot D\iota(x)(v).$$
Since $\ell$ is a non-constant curve in $\G$, it must be that $\pi\circ \ell$ is non-constant (Lemma \ref{pifacts}\eqref{picurve}). Hence, $D\iota(x)(v) \neq 0$. As $v$ was arbitrary, we conclude that $D\iota(x)$ is injective, and in particular that $k \leq n$.

Given $i_1 < \dots < i_k\in\{1, \dots, n\}$, let $U_{i_1,\ldots,i_k}$ denote the subset of $U$ on which the $k\times k$ minor of $(D\iota)(x)$ defined by the coordinates $i_1, \dots, i_k$ is invertible. Note that almost every $x\in U$ is in some such set, since $(D\iota)(x)$ is injective for almost every $x\in U$. For $x\in U_{i_1, \dots, i_k}$, let $A(x) \colon \R^n\rightarrow\R^k$ be the unique linear mapping with 
\begin{equation}\label{Akernel}
 \text{ker}(A)= \text{span}(\{e_1, \dots, e_n\} \setminus \{e_{i_1}, \dots, e_{i_k}\})
\end{equation}
such that 
\begin{equation}\label{Ainverse}
A(x) \cdot (D\iota)(x) = \text{Id}_{k}.
\end{equation}
Applying \eqref{Ainverse} to equation \eqref{piderivative}, we see that
\begin{equation}\label{Aphipi}
\phi(y) - \phi(x) = A(x) \cdot (\pi(y) - \pi(x)) + o(d(x,y)), \hspace{0.3cm} y \in X.
\end{equation}
From \eqref{Aphipi} and \eqref{Akernel}, it follows that a Lipschitz function $f\colon X\rightarrow\R$ that is differentiable with respect to $\phi$ at $x\in  U_{i_1, \dots, i_k}$, with unique derivative, is differentiable with respect to the restrictions $(\pi_{i_1}|_X, \ldots, \pi_{i_k}|_X)$ with unique derivative.

Hence, these form a $k$-dimensional chart map for $U_{i_1, \dots, i_k}$.
\end{proof}

We are now ready to prove Theorem \ref{tan}.

\begin{proof}[\textbf{Proof of Theorem \ref{tan}}]

Using the previous lemma, and passing to subsets, we may assume without loss of generality that $U$ has chart map
$$(\pi_1,\ldots,\pi_k) \colon X \rightarrow \R^k \subset \R^n.$$
For almost every $x \in U$, we know that the other coordinates $\pi_{k+1},\ldots,\pi_n$ are linear combinations of $\pi_1,\ldots,\pi_k$, up to first order on $X$ near $x$. Fix such a point $x$, and assume also that every $(\hat{X},\hat{x})\in \Tan(X,x)$ is rectifiably connected and all tangent maps $(\hat{\pi}_1,\ldots,\hat{\pi}_k) \colon \hat{X} \rightarrow \R^k$ are Lipschitz quotients. By assumption, such points form a set of full measure in $U$.

Fix $(\hat{X}, \hat{x}) \in \Tan(X,x)$. We canonically identify $\hat{X}$ with a closed subset of $\G$ that is a limit of rescalings of $L_{-x}(X)$, with $\hat{x}=0$ and $\hat{\pi}_i = \pi_i$. Moreover, as the coordinates $\pi_{k+1},\ldots,\pi_n$ were linear combinations of $\pi_1,\ldots,\pi_k$, up to first order on $X$ near $x$, we see that $\pi_{k+1}|_{\hat{X}},\ldots,\pi_n|_{\hat{X}}$ are precisely linear combinations of $\pi_1|_{\hat{X}},\ldots,\pi_k|_{\hat{X}}$ on $\hat{X}$. In particular, there is an injective linear transformation $A \colon \R^k \rightarrow \R^n$ for which
$$p := \pi |_{\hat{X}} = A \circ (\pi_1, \ldots, \pi_k).$$
Let $V = p(\hat{X}) = A(\R^k)$ be the corresponding linear subspace of $\R^n \simeq V_1$. It is a $k$-dimensional subspace because $(\pi_1,\ldots, \pi_k) \colon \hat{X} \rightarrow \R^k$ is surjective by the Lipschitz quotient property and $A$ is injective. Furthermore, the map $p\colon \hat{X}\rightarrow V$ is a Lipschitz quotient.

We claim that $\hat{X}$ is the Carnot subgroup of $\G$ generated by $V \subset \R^n$. First, we show that it is a subgroup. Let $y,z \in \hat{X}$, and let $\gamma$ be a rectifiable curve in $\hat{X}$ from $0$ to $z$, so that $p \circ \gamma$ is a rectifiable curve in $V$ from $0$ to $p(z)$. Consider the rectifiable curve $L_y \circ \gamma$, which joins $y$ to $yz$ in $\G$. As
$$\pi \circ L_y \circ \gamma = L_{\pi(y)} \circ \pi \circ \gamma = L_{p(y)} \circ p \circ \gamma,$$
we see that $L_y \circ \gamma$ is the unique horizontal lift through $\pi$ of the curve $L_{p(y)} \circ p \circ \gamma$ to $\G$, starting at $y$ (recall the uniqueness in Lemma \ref{pifacts}\eqref{pilift}). Note that $L_{p(y)} \circ p \circ \gamma$ is a rectifiable curve in $V$, as $L_{p(y)}$ preserves this subspace. Applying Lemma \ref{pathlifting} to the Lipschitz quotient $p \colon \hat{X} \rightarrow V$, we obtain a lift of $L_{p(y)} \circ p \circ \gamma$ to a rectifiable curve in $\hat{X}$ beginning at $y$. This lift through $p$ is also a lift through $\pi$, and by uniqueness of horizontal lifts to $\G$, the two lifts must coincide. Hence, $L_y \circ \gamma$ is contained in $\hat{X}$; in particular, its endpoint $yz$ is in $\hat{X}$.

Similar arguments show that $\hat{X}$ is closed under group inversion and dilations. If $y \in \hat{X}$, let $\gamma$ be a rectifiable curve from $0$ to $y$ in $\hat{X}$. Then $\gamma^{-1}$ is also a rectifiable curve in $\G$, and 
$$\pi \circ \gamma^{-1} = - \pi \circ \gamma = -p \circ \gamma$$
is a rectifiable curve in $V$. We see that $\gamma^{-1}$ is the unique horizontal lift of $-p \circ \gamma$ to $\G$ which begins at $0$. The Lipschitz quotient property for $p$ ensures that such a lift must exist in $\hat{X}$, so in fact $\gamma^{-1}$ lies in $\hat{X}$. This gives $y^{-1} \in \hat{X}$. Finally, if $\delta_{\lambda}$ is a dilation for $\G$, then $\delta_{\lambda} \circ \gamma$ is again a horizontal curve in $\G$. As
$$\pi \circ \delta_{\lambda} \circ \gamma = \lambda \cdot \pi \circ \gamma = \lambda \cdot p \circ \gamma$$
is a rectifiable curve in $V$, its unique horizontal lift to $\G$ that starts at $0$ is $\delta_{\lambda} \circ \gamma$. By the Lipschitz quotient property for $p$, this lift must be in $\hat{X}$, so $\delta_{\lambda}(y) \in \hat{X}$.

Thus we find that $\hat{X}=H$ is a homogeneous subgroup of $\G$ that is rectifiably connected. In particular, it is a closed Lie subgroup with Lie algebra $\mathfrak{h} \subset \mathfrak{g}$ for which $H = \exp_{\G}(\mathfrak{h})$. Moreover, the $k$-dimensional subspace $V = p(H) \subset V_1$ that we found earlier is precisely $\mathfrak{h} \cap V_1$. Let $\mathfrak{h}^* \subset \mathfrak{h}$ denote the stratified Lie algebra generated by $V$, and let $H^* = \exp_{\G}(\mathfrak{h}^*) \subset H$ be the corresponding Carnot subgroup of $\G$. 

It remains only to show that $H^* = H$. To see this, let $y \in H$ be arbitrary, and let $\gamma$ be a rectifiable curve in $H$ from $0$ to $y$. Then $p \circ \gamma$ is a rectifiable curve in $V$ from $0$ to $p(z)$, and its unique horizontal lift to $\G$, starting at $0$, is $\gamma$. At the same time, $p \circ \gamma$ is a rectifiable curve in the horizontal layer of the Carnot subgroup $H^*$, so it has a horizontal lift to $H^*$ through $p$ that starts at $0$. Hence $y \in H^*$, and we conclude that $H^* = H$.
\end{proof}

The following is an immediate result of the theorem above.

\begin{corollary}\label{embed}
Let $(U,\phi)$ be a $k$-dimensional chart in a complete, metrically doubling Lipschitz differentiability space $(X,d,\mu)$, and assume that for almost every $x \in U$, each $\hat{X} \in \Tan(X,x)$ is rectifiably connected. 

Suppose that $F \colon X \rightarrow \G$ is a bi-Lipschitz embedding. Then for $\mu$-almost every $x \in U$, $\Tan(F(X), F(x))$ consists of one element, which is canonically isometric to a Carnot subgroup of $\G$ generated by a $k$-dimensional vector subspace of the horizontal layer of $\G$.

In particular, for $\mu$-almost every $x\in U$, every element in $\Tan(X,x)$ is bi-Lipschitz equivalent to a sub-Riemannian Carnot group whose horizontal layer has dimension $k$.
\end{corollary}

\subsection{Non-embedding consequences}\label{nonembedding}

One interesting consequence of Cheeger's initial study of differentiability in metric spaces was a certain generalized non-embedding result for Euclidean targets. The following statement was proven for PI spaces in \cite[Theorems 14.1 and 14.2]{Ch99} and was generalized to Lipschitz differentiability spaces in \cite[Corollary 8.3]{GCD15} as a corollary of Proposition \ref{LQ} above.

\begin{theorem}\label{nonembedRn}
Let $(X,d,\mu)$ be a complete, metrically doubling Lipschitz differentiability space with an $k$-dimensional chart $(U,\phi)$. 

Suppose there exists a set $A \subseteq U$ with $\mu(A)>0$ such that for every $a\in A$, there exists $(Y,y)\in \Tan(X,a)$ that is \textit{not} bi-Lipschitz equivalent to $\mathbb{R}^k$. 

Then $X$ does not admit a bi-Lipschitz embedding into any Euclidean space.
\end{theorem}

In \cite[Theorem 1.6]{CK09}, Cheeger and Kleiner used differentiability to prove a non-embedding result for PI spaces into certain infinite-dimensional Banach spaces.

Our work above allows us to obtain a result similar to Theorem \ref{nonembedRn} for Carnot group targets. Namely, a direct consequence of Corollary \ref{embed} is the following general non-embedding result.

\begin{corollary}\label{nonembed}
Let $(X,d,\mu)$ be a complete, metrically doubling Lipschitz differentiability space with a $k$-dimensional chart $(U,\phi)$, and assume that for almost every $x \in U$, each $\hat{X} \in \Tan(X,x)$ is rectifiably connected.

Suppose there is a set $A\subseteq U$ with $\mu(A)>0$ such that, for every $a\in A$, there exists $(Y,y)\in \Tan(X,a)$ that is \textit{not} bi-Lipschitz equivalent to a Carnot group with a $k$-dimensional horizontal layer.
 
Then $X$ does not admit a bi-Lipschitz embedding into any Carnot group.
\end{corollary}

A non-embedding result in the same spirit (but allowing more general targets) for a class of one-dimensional Lipschitz differentiability spaces already appears in \cite[Corollary 9.3]{CKS15}.

\section{Rigidity for convex-cocompact actions of minimal Hausdorff dimension}\label{mainproof}
This section is devoted to the proof of Theorem \ref{mainthm}. It will be convenient to begin with the following lemma, which is essentially due to Bourdon, as part of the argument for Theorem \ref{Bouthm}.

\begin{lemma}\label{bourdonlemma}
Let $\Gamma \acts X$ be a convex-cocompact action of a discrete group $\Gamma$ on a $\CAT(-1)$ space $X$. Suppose that the limit set $\Lambda(\Gamma)$ is bi-Lipschitz equivalent to the boundary of a non-compact rank-one symmetric space $S \neq \Hy_\R^2$. Then there is an isometric embedding $F \colon S \rightarrow X$ such that $F(\partial S) = \Lambda(\Gamma)$, and $\Gamma$ stabilizes $F(S)$, acting on it with compact quotient.
\end{lemma}

\begin{proof}
Fix visual metrics on $\partial X$ and $\partial S$ once and for all. The bi-Lipschitz equivalence implies that $\dim_H(\Lambda(\Gamma)) = \dim_H(\partial S)$ and that the M\"obius action $\Gamma \acts \Lambda(\Gamma)$ can be conjugated to a uniformly quasi-M\"obius action $\Gamma \acts \partial S$.

If $S$ is of quaternionic or octonionic type, then a theorem of Pansu \cite[Corollary 11.2]{Pa89} ensures that the action $\Gamma \acts \partial S$ is by $1$-quasiconformal mappings, hence mappings that extend to isometries of $S$ by \cite[Theorem 11.5]{Pa89}, and hence M\"obius mappings by \cite{Bou95}.

If $S$ is real or complex hyperbolic, then theorems of Sullivan--Tukia \cite[Theorem G]{Tu86} and Chow \cite[Theorem 2]{Ch96} show that there is a quasi-M\"obius homeomorphism of $\partial S$ that conjugates the quasi-M\"obius action $\Gamma \acts \partial S$ to a M\"obius action. To be more precise, here we use that every point in $\partial S$ is a ``radial limit point" for the the quasi-M\"obius action $\Gamma \acts \partial S$. This follows from the fact that this action is cocompact on triples, which in turn follows from the standard fact that the boundary action $\Gamma \acts \Lambda(\Gamma)$ is cocompact on triples. Also, in the complex hyperbolic case, we again turn to \cite[Theorem 11.5]{Pa89} and \cite{Bou95} to argue that conformal mappings are M\"obius.

In any case, we can find a quasisymmetric homeomorphism $f \colon \partial S \rightarrow \Lambda(\Gamma)$ that is equivariant with respect to M\"obius actions $\Gamma \acts \partial S$ and $\Gamma \acts \Lambda(\Gamma)$. 

From here, Bourdon's work applies directly. In Section 2 of \cite{Bou}, he shows that such a homeomorphism $f$ is in fact a M\"obius homeomorphism. In addition, he shows in \cite[Theorem 0.1]{Bou} that any M\"obius embedding $f\colon \partial S \rightarrow \partial X$ extends to an isometric embedding $F \colon S \rightarrow X$. That $\Gamma \acts X$ stabilizes $F(S)$ follows immediately from the fact that $F(S)$ is the union of geodesics in $X$ whose endpoints are both in $\Lambda(\Gamma)$. That $F(S)/\Gamma$ is compact follows from the assumption that $\Gamma \acts X$ is convex-cocompact.
\end{proof}

\begin{proof}[\textbf{Proof of Theorem \ref{mainthm}}]
The inequality was already discussed in Remark \ref{tfaermk}, as was the equality case when $Q \leq 1$. Thus, we may assume that equality holds and $Q > 1$. It follows from Theorem \ref{tfae} that $\Lambda(\Gamma)$, equipped with $Q$-dimensional Hausdorff measure, is a Lipschitz differentiability space. In fact, we know that $\Lambda(\Gamma)$ supports a Poincar\'e inequality, which means that it is a quasiconvex metric space (cf. Remark \ref{tanrmk}).

Let $\G$ be the Iwasawa group that locally models the geometry of $(\partial X,d_p)$. By Lemma \ref{paralemma}, for any choice of $q \in X$ and $\omega \in \partial X$, the punctured boundary $\partial X \backslash \{\omega\}$ is identified with $\G$ in such a way that the parabolic visual metric $d_{\omega,q}$ is a homogeneous distance on $\G$.

Choose a point $z \in \Lambda(\Gamma) \subset \partial X$ for which every tangent of every element of $\Tan(\Lambda(\Gamma),d_p,z)$ is also in $\Tan(\Lambda(\Gamma),d_p,z)$. By \cite[Theorem 1.1]{LDtan} almost every $z \in \Lambda(\Gamma)$, with respect to Hausdorff $Q$-measure, has this property. Take a tangent of $(\Lambda(\Gamma),d_p)$ at $z$. By \cite[Proposition 3.1]{BKqf}, there are points $q' \in X$ and $\omega' \in \Lambda(\Gamma)$ for which this tangent is isometric to the parabolic limit set $(\Lambda(\Gamma) \backslash \{\omega'\}, d_{\omega',q'})$. As $(\Lambda(\Gamma),d_p)$ is quasiconvex, its tangents are as well.

This parabolic limit set is a subset of the parabolic boundary $(\partial X \backslash \{\omega'\}, d_{\omega',q'})$, which is identified with the Iwasawa group $\G$. The metrics $d_{\omega',q'}$ and $d_p$ are locally bi-Lipschitz equivalent on $\partial X \backslash \{\omega'\}$ and, hence, also on $\Lambda(\Gamma) \backslash \{\omega'\}$. In particular, this means that $(\Lambda(\Gamma) \backslash \{\omega'\}, d_{\omega',q'})$ is a Lipschitz differentiability space when equipped with its Hausdorff $Q$-measure. By Theorem \ref{tan}, this parabolic limit set has a tangent that is isometric to a Carnot subgroup $\N \subset \G$.

By the way we chose $z \in \Lambda(\Gamma)$, we see that $(\N,d_{\omega',q'})$ is actually a tangent of the full limit set $(\Lambda(\Gamma),d_p)$. Once again, appealing to \cite[Proposition 3.1]{BKqf}, this means that there are points $q \in X$ and $\omega \in \Lambda(\Gamma)$ for which $(\N,d_{\omega',q'})$ is isometric to the parabolic limit set $(\Lambda(\Gamma) \backslash \{\omega\}, d_{\omega,q})$. Fix an identification of $\partial X \backslash \{\omega\}$ with $\G$ so that the origin of $\G$ is contained in $\Lambda(\Gamma) \backslash \{\omega\}$. In this way, $\N = \Lambda(\Gamma) \backslash \{\omega\}$ is a Carnot subgroup of $\G = \partial X \backslash \{\omega\}$, which is equipped with the metric $d_{\omega,q}$.

We claim that $\N$ is also an Iwasawa group. Suppose not, and let $g \in \Gamma$ be arbitrary. 
As remarked in Section \ref{visualsec}, the induced boundary homeomorphism
$$g \colon \G \backslash \{g^{-1}(\omega)\} \rightarrow \G \backslash \{g(\omega)\}$$
is conformal, in the sense that it is smooth and its horizontal derivative is everywhere a similarity. In particular, the restriction
$$g \colon \N \backslash \{g^{-1}(\omega)\} \rightarrow \N \backslash \{g(\omega)\}$$
is also smooth with horizontal derivative everywhere a similarity, and so it is a conformal map on an open, connected subset of $\N$. By \cite[Theorem 4.1]{CO15}, this means that $g|_{\N \backslash \{g^{-1}(\omega)\}}$ is the restriction of an affine map, and in particular that $g(\omega) = \omega$. Thus, $\omega$ is a global fixed point for the action of $\Gamma$ on $\Lambda(\Gamma)$, a contradiction.

Let $S$ be the non-compact rank-one symmetric space with $\partial S$ locally modeled by the Iwasawa group $\N$. Note that $S$ is determined uniquely by the topological and Hausdorff dimensions of $\N$, or equivalently of $\Lambda(\Gamma)$ \cite[p. 34]{MTcdim}. Choosing any points $s \in S$ and $\zeta \in \partial S$, we know that $(\partial S \backslash \{\zeta\}, d_{\zeta,s})$ is bi-Lipschitz equivalent to $\N = (\Lambda(\Gamma) \backslash \{\omega\}, d_{\omega,q})$. 

Using that $(\partial S \backslash \{\zeta\}, d_{\zeta,s})$ and $(\partial S \backslash \{\zeta\}, d_s)$ are M\"obius equivalent, together with the fact that $(\Lambda(\Gamma) \backslash \{\omega\}, d_{\omega,q})$ and $(\Lambda(\Gamma) \backslash \{\omega\}, d_q)$ are M\"obius equivalent, we find that $(\partial S \backslash \{\zeta\}, d_s)$ is quasi-M\"obius equivalent to $(\Lambda(\Gamma) \backslash \{\omega\}, d_q)$ with linear distortion function $\eta(t)$. In particular, this implies that $(\partial S,d_s)$ and $(\Lambda(\Gamma),d_q)$ are bi-Lipschitz equivalent. As $\Hdim(\partial S) = \Hdim(\Lambda(\Gamma)) = Q >1$, we know that $S \neq \Hy_\R^2$. Finally, applying Lemma \ref{bourdonlemma} finishes the proof.
\end{proof}

\bibliography{carnot}{}
\bibliographystyle{plain}

\end{document}